\documentclass[12pt,a4paper]{amsart} 
\usepackage[utf8]{inputenc} 
\usepackage[top=3cm, bottom=3.5cm, left=2.5cm, right=2.5cm]{geometry}
\usepackage{amsmath,amsthm,amsfonts,amstext,amssymb}
\usepackage{mathrsfs}
\usepackage{enumerate}
\usepackage{hyperref}
\usepackage{dsfont}
\usepackage{color}
\usepackage{graphicx}
\usepackage{xcolor}
\hypersetup{
    colorlinks,
    linkcolor={blue!50!black},
    citecolor={blue!50!black},
    urlcolor={blue!80!black}
}

\def\R{{\mathbb{R}}}

\def\Z{{\mathbb{Z}}}

\newtheorem{theorem}{Theorem}[section]

\newtheorem{lemma}[theorem]{Lemma}
\newtheorem{proposition}[theorem]{Proposition}

\theoremstyle{definition}

\newtheorem{remark}[theorem]{Remark}

\numberwithin{equation}{section}

\begin{document}

\title[Visibility in Brownian interlacements, Poisson cylinders and Boolean models]{Visibility in Brownian interlacements, Poisson cylinders and Boolean models}
\author{Yingxin Mu}
\address{
  Yingxin Mu,
  University of Leipzig, Institute of Mathematics,
  Augustusplatz 10, 04109 Leipzig, Germany.
}
\email{yingxin.mu@uni-leipzig.de}

\author{Artem Sapozhnikov}
\address{
  Artem Sapozhnikov,
  University of Leipzig, Institute of Mathematics,
  Augustusplatz 10, 04109 Leipzig, Germany.
}
\email{artem.sapozhnikov@math.uni-leipzig.de}

\begin{abstract}
We study visibility inside the vacant set of three models in $\R^d$ with slow decay of spatial correlations: Brownian interlacements, Poisson cylinders and Boolean model. For each of them, we obtain sharp asymptotic bounds on the probability of visibility to distance $r$ in some direction in terms of the probability of visibility to distance $r$ in a given direction. In dimensions $d\geq 4$, the ratio of the two probabilities has the same scaling $r^{2(d-1)}$ for all three models, but in lower dimensions the scalings are different.
In particular, we improve some main results from \cite{Calka-visibility,ET-visibility}.
\end{abstract}

 
\maketitle

\section{Introduction}
The problem of visibility in random fields of obscuring elements is of interest in various applications, see \cite{Zacks-visibility}. The first mathematical study of visibility goes back to Pólya in 1918 \cite{Polya-visibility}. He showed that in a forest of trees of constant radius $R$ planted in the vertices of the lattice $\Z^2$, an unobstructed visibility to distance $r$ from the origin is possible if $R=O(\tfrac1r)$ when $r\to\infty$. 
In \cite{Calka-visibility}, Calka et al.\ considered visibility through the vacant set of Boolean model in $\R^d$ with fixed radius of balls; they obtained an explicit formula for the probability of visibility in dimension $2$ and gave asymptotic bounds in dimensions $d\geq 3$. Recently, Elias and Tykesson \cite{ET-visibility} studied visibility through the vacant set of Brownian interlacements in $\R^d$ ($d\geq 3$) resp.\ of Brownian excursion set in the unit disc. For the former, they obtained bounds on the probability of visibility to distance $r$ in terms of the probability of visibility to distance $r$ in a given direction. 
Visibility in the hyperbolic plane was studied in \cite{Kahane-1,Kahane-2,Kahane-3,BJST-visibility-H,TC-visibility-H}.

\smallskip

In this article, we follow up \cite{Calka-visibility} and \cite{ET-visibility} and study visibility through the vacant set of three models in $\R^d$ with slow (algebraic) decay of spatial correlations: 
\emph{Brownian interlacements}, \emph{Poisson cylinders} and \emph{Boolean model}. For all three models, we obtain sharp asymptotic bounds on the probability of visibility to distance $r$ in terms of the probability of visibility to distance $r$ in a given direction; in particular, we improve some main results from \cite{Calka-visibility,ET-visibility}.

\smallskip

In a general setup, let $\mathcal C$ be a random closed subset of $\R^d$ with distribution $\mathsf P$, which is rotationally invariant.\footnote{$\mathsf P$ is a probability measure on measurable space $(\Sigma,\mathscr F)$, where $\Sigma$ is the set of closed sets in $\R^d$ and $\mathscr F$ is the sigma-algebra on $\Sigma$ generated by sets $\{F\in\Sigma\,:\,F\cap K=\emptyset\}$ for compacts $K$.} We are interested in the visibility through the vacant set of $\mathcal C$, defined as $\mathcal V=\R^d\setminus \mathcal C$. We write $x\stackrel{L}\longleftrightarrow y$ if the line segment $[x,y]$ is contained in $\mathcal V$ and say that $y$ is \emph{visible} from $x$. We also write $x\stackrel{L}\longleftrightarrow A$ for $x\in\R^d $ and $A\subseteq\R^d$, if there exist $y\in A$ such that $x\stackrel{L}\longleftrightarrow y$ and say that $A$ is visible from $x$. For $r>0$, let 
\begin{equation}\label{def:fPvis}
f(r) = \mathsf P[0\stackrel{L}\longleftrightarrow re_1]\quad\text{and}\quad 
P_{\mathrm{vis}}(r) = \mathsf P[0\stackrel{L}\longleftrightarrow \partial B(0,r)]
\end{equation}
be the probability of visibility to distance $r$ in a given direction resp.\ to distance $r$ in some direction. 
Our main result gives bounds on $P_{\mathrm{vis}}(r)$ in terms of $f(r)$,
\begin{equation}\label{intro:visibility-bounds}
c_1\big(\tfrac{1}{\delta(r)}\big)^{d-1}r^{d-1}f(r)\leq P_{\mathrm{vis}}(r)\leq 
c_2\big(\tfrac{1}{\delta(r)}\big)^{d-1}r^{d-1}f(r),
\end{equation}
where $c_i=c_i(d,\mathsf P)$, in each of the three models, which we now introduce informally. 

\smallskip

 \emph{Brownian interlacements} is the range of a Poisson cloud of doubly infinite Wiener sausages of fixed radius $\rho$ in $\R^d$ ($d\geq 3$), whose density is governed by a parameter $\alpha>0$. It was introduced by Sznitman in \cite{Sznitman-BI} as the continuous analogue of random interlacements (see \cite{Sznitman-AM}). For any positive $\alpha$ and $\rho$, it is an almost surely connected random subset of $\R^d$, see \cite{Li-BI}. Its vacant set, however, undergoes a non-trivial percolation phase transition, see \cite{Li-BI,MS-BI}.
 \emph{Poisson cylinders} is the range of a Poisson cloud of doubly infinite cylinders of fixed radius $\rho$ in $\R^d$ ($d\geq 2$), whose density is governed by a parameter $\alpha>0$. It was introduced by Tykesson and Windisch in \cite{TW-cylinders} on suggestion of Benjamini.  As in the case of Brownian interlacements, the occupied and the vacant set of Poisson cylinders have quite different properties. The occupied set is almost surely connected for all positive $\alpha$ and $\rho$, see \cite{BT-PC-connected}. The vacant set undergoes a non-trivial percolation phase transition in dimensions $d\geq 3$, see \cite{TW-cylinders,HST-PC-d3}.
 \emph{Boolean model} is the range of a Poisson cloud of Euclidean balls with random radii in $\R^d$ ($d\geq 2$), whose density is governed by a parameter $\alpha>0$. It is a classical model of continuum percolation and has been comprehensively studied, see e.g.\ \cite{MR-Book, Gouere08} and more recent \cite{DRT-Boolean,ATT18,Pen18} for developments on the phase transition in the occupied and the vacant set. 

\smallskip

We refer to seminal papers \cite{Sznitman-BI,TW-cylinders} and to \cite{MR-Book} for the precise constructions of these models in terms of Poisson point processes and their basic properties. In fact, for the purpose of this article, we only need the characterization of the laws of these random sets by the functional 
\[
T(K) = \mathsf P[\mathcal C\cap K\neq\emptyset],\quad K\subset\R^d\text{ compact}\footnote{For more on this topic, we refer to \cite[Chapter~2]{Matheron}.}
\]
(and monotonicity in the radius $\rho$ for the upper bound in \eqref{intro:visibility-bounds}), see \eqref{eq:BI-capacity-rho}, \eqref{eq:PC-mu-rho} and \eqref{eq:BM-mu}. 

\smallskip

We now present our main result. 
\begin{theorem}\label{intro:visibility-mainresult}
Let $\alpha>0$ and $\rho>0$. Let $\mathsf Q$ be a probability measure on $\R_+$. 
\begin{enumerate}
\item
Brownian interlacements at level $\alpha$ with radius $\rho$ in $\R^d$ ($d\geq 3$) satisfies the bounds \eqref{intro:visibility-bounds} for all $r\geq 2$ with some $c_i=c_i(d,\alpha,\rho)$ and 
\[
\delta(r) = \left\{\begin{array}{ll} r^{-1} & d\geq 4\\[4pt] r^{-1}\log^2r & d=3\end{array}\right.
\]
\item
Poisson cylinders at level $\alpha$ with radius $\rho$ in $\R^d$ ($d\geq 2$) satisfies the bounds \eqref{intro:visibility-bounds} for all $r\geq 1$ with some $c_i=c_i(d,\alpha,\rho)$ and 
\[
\delta(r) = \left\{\begin{array}{ll} r^{-1} & d\geq 3\\[4pt] 1 & d=2\end{array}\right.
\]
\item
Boolean model with intensity $\alpha$ and radii distribution $\mathsf Q$ in $\R^d$ ($d\geq 2$) satisfies the bounds \eqref{intro:visibility-bounds} for all $r\geq 1$ with some $c_i=c_i(d,\alpha,\mathsf Q)$ and 
\[
\delta(r) = r^{-1}.
\]
\end{enumerate}
\end{theorem}
One could think about $\delta(r)$ in Theorem~\ref{intro:visibility-mainresult} as a \emph{visibility window} on $\partial B(0,r)$: in essence, the conditional probability that $y\in\partial B(0,r)$ is visible from $0$, given that some $x\in\partial B(0,r)$ is, is positive uniformly in $r$ when $\|x-y\|<\delta(r)$, and smaller than $\exp(-c\frac{1}{\delta(r)}\|x-y\|)$ when $\delta(r)<\|x-y\|<1$, cf.\ Proposition~\ref{prop:visibility-general}. This may explain the bounds~\eqref{intro:visibility-bounds}, since $\big(\frac{r}{\delta(r)}\big)^{d-1}f(r)$ is the order of the expected number of balls of radius $\delta(r)$ in a minimal covering of $\partial B(0,r)$, whose centers are visible from $0$. In this regard, it is interesting to see different scalings for different models and different dimensions. In the case of Brownian interlacements, the inverse of $\delta(r)$ equals (up to a constant factor) to the product of the Newtonian capacity of a cylinder with length $r$ and fixed radius (see \eqref{def:capacity} and \eqref{eq:capacity-cylinder}) and the probability that a Brownian motion avoids the cylinder, when started at a fixed distance from it; in particular, in dimension $d=3$, one of the $\log$'s comes from the capacity and the second one from the non-hitting probability (see \eqref{eq:capacity-cylinder} and \eqref{eq:BI-lowerbound-main}).

\smallskip

Let us compare our results to previous works. 
Partial result about Brownian interlacements was obtained by Elias and Tykesson \cite{ET-visibility}. They proved that (a) $P_{\mathrm{vis}}(r)\leq Cr^{2(d-1)}f(r)$ in dimensions $d\geq 3$ and (b) $P_{\mathrm{vis}}(r)\geq cr^{d-1}f(r)$ in dimensions $d\geq 4$, see \cite[Theorem~2.3]{ET-visibility}. In particular, their upper bound turns to be sharp in dimensions $d\geq 4$, which was not apparent from their analysis. 
In the setting of Boolean model with constant radii, Calka et al.\ \cite{Calka-visibility} were interested in the logarithmic asymptotics of $P_{\mathrm{vis}}(r)$; they obtained exact limit in $d=2$ and upper and lower bounds with different constants in $d\geq 3$, see \cite[Proposition~2.3]{Calka-visibility}. 
In Lemma~\ref{l:BM-fr}, we compute for general Boolean models that 
$f(r) = \exp\big(-\alpha(\kappa_d\mathsf E[\rho^d]+\kappa_{d-1}\mathsf E[\rho^{d-1}]r)\big)$, where $\kappa_s$ is the volume of the $s$-dimensional Euclidean unit ball and $\rho$ is a generic random variable with distribution $\mathsf Q$. Thus, by Theorem~\ref{intro:visibility-mainresult}(3), if $\mathsf E[\rho^d]<\infty$, then
\[
\lim\limits_{r\to\infty}\frac1r\log P_{\mathrm{vis}}(r) = \lim\limits_{r\to\infty}\frac1r\log f(r) = -\alpha \kappa_{d-1}\mathsf E[\rho^{d-1}],
\]
which sharpens the upper bound in \cite[Proposition~2.3]{Calka-visibility} and extends the result to Boolean models with arbitrary radii distributions.

\smallskip

The rest of the paper is organized as follows. In Section~\ref{sec:general-conditions} we fix common notation used in the proofs and introduce some general conditions on correlations, which imply bounds \eqref{intro:visibility-bounds}, see Proposition~\ref{prop:visibility-general}. We then define Brownian interlacements, Poisson cylinders and Boolean models precisely and prove---by verifying conditions of Propostion~\ref{prop:visibility-general}---the three parts of Theorem~\ref{intro:visibility-mainresult}, respectively, in Sections~\ref{sec:BI}, \ref{sec:PC} and \ref{sec:BM}.

\section{Notation and general conditions}\label{sec:general-conditions}
In this section, we introduce some common notation, used throughout the paper, and discuss general conditions, which imply bounds \eqref{intro:visibility-bounds}. 

\smallskip

Let $x\in\R^d$, $\rho>0$ and $K,K'\subset \R^d$. We denote by $B(x,\rho)$ the closed Euclidean ball at $x$ with radius $\rho$ and write $B(\rho)=B(0,\rho)$. We denote the closed $\rho$-neighborhood of $K$ by $B(K,\rho$), that is $B(K,\rho) = \bigcup_{x\in K}B(x,\rho)$. 
We denote by $\ell_x$ the line segment $[0,x]$ in $\R^d$ and define $\ell_x(\rho) = B(\ell_x,\rho)$. For convenience, we refer to $\ell_x(\rho)$ as (finite) cylinder of radius $\rho$ around $\ell_x$. 
The Euclidean distance between $K$ and $K'$ is denoted by $d(K,K')$ and we write $d(x,K)$ instead of $d(\{x\},K)$ for the distance from $x$ to $K$. 
We write $e_1$ for the first vector in the canonical basis of $\R^d$. 

\smallskip

The general first step in the proof of the relations \eqref{intro:visibility-bounds} between $P_{\mathrm{vis}}(r)$ and $f(r)$ is an application of the first (for the upper bound) and the second (for the lower bound) moment methods. To avoid repetitions, we summarize this in the next proposition. 

\begin{proposition}\label{prop:visibility-general}
Let $\mathcal C$ be a random rotationally invariant closed set in $\R^d$. 

There exist constants $c=c(d)>0$ and $C=C(d)<\infty$, such that for any $\delta>0$, $r>0$, $s\in(0,2r]$ and $D<\infty$,  
\begin{enumerate}
\item
(Upper bound) if 
\begin{equation}\label{eq:visibility-general-upper}
\mathsf P\big[0\stackrel{L}\longleftrightarrow \partial B(r)\cap B(re_1,\delta)\big]
\leq D f(r),
\end{equation}
then 
\[
P_{\mathrm{vis}}(r)\leq CD \big(\tfrac{r}{\delta}\big)^{d-1}f(r).
\]
\item
(Lower bound) if
\begin{equation}\label{eq:visibility-general-lower}
\mathsf P\big[0\stackrel{L}\longleftrightarrow x\,|\,0\stackrel{L}\longleftrightarrow re_1\big]\leq D\exp\big(-\tfrac{1}{\delta}\min(\|x-re_1\|,1)\big),\quad\text{for all $x\in\partial B(r)\cap B(re_1,s)$,}
\end{equation}
then 
\[
P_{\mathrm{vis}}(r)\geq cD^{-1} \big(\delta^{d-1}+s^{d-1}e^{-\frac1\delta}\big)^{-1}s^{d-1}f(r).\footnote{In most applications, $s$ is chosen as a multiple of $r$ and $\delta$ as a function of $r$, which decays polynomially to $0$; in particular, $s^{d-1}e^{-\frac1\delta}=o(\delta^{d-1})$. See also Remark~\ref{rem:visibility-general} for a variation of part 2.\ of Proposition~\ref{prop:visibility-general}.}
\]
\end{enumerate}
\end{proposition}
The proof of part (1) is essentially the same as \cite[(4.17)]{ET-visibility}. The proof of part (2) is a refinement of the second moment argument in \cite[Section~4.1]{ET-visibility}.
\begin{proof}
Assume \eqref{eq:visibility-general-upper}. 
There is a covering of $\partial B(r)$ by at most $C\big(\tfrac{r}{\delta}\big)^{d-1}$ balls of radius $\delta$ centered on $\partial B(r)$ for some $C=C(d)$. 
If $\partial B(r)$ is visible from $0$, then the part of the boundary covered by one of the balls is visible from $0$. 
Thus, by rotational invariance of $\mathcal C$, 
\[
P_{\mathrm{vis}}(r) \leq C\big(\tfrac{r}{\delta}\big)^{d-1}\,
\mathsf P\big[0\stackrel{L}\longleftrightarrow \partial B(r)\cap B(re_1,\delta)\big]
\stackrel{\eqref{eq:visibility-general-upper}}\leq 
CD \big(\tfrac{r}{\delta}\big)^{d-1}f(r),
\]
which proves the first part of the proposition. 

\smallskip

Now, assume \eqref{eq:visibility-general-lower}. Consider spherical caps $S=\partial B(r)\cap B(re_1,\frac s2)$ and $S' = \partial B(r)\cap B(re_1,s)$. Let $\sigma(\cdot)$ be the surface measure on $\partial B(r)$. Note that $c_1 s^{d-1}\leq \sigma(S)\leq \sigma(S') \leq c_2 s^{d-1}$, for some $c_1=c_1(d)$ and $c_2=c_2(d)$. 
Consider 
\[
X=\int\limits_S\mathds{1}_{\{0\stackrel{L}\longleftrightarrow x\}}\sigma(dx)
\]
and note that $P_{\mathrm{vis}}(r)\geq \mathsf P [X>0]$.
By Fubini and rotational invariance of $\mathcal C$, $\mathsf E[X] = \sigma(S)f(r) \geq c_1 s^{d-1}f(r)$ and 
\begin{eqnarray*}
\mathsf E[X^2] &= &\int\limits_{S\times S} \mathsf P[0\stackrel{L}\longleftrightarrow x,0\stackrel{L}\longleftrightarrow y]\sigma(dx)\sigma(dy)\\
&\leq &\sigma(S)f(r)\,\int\limits_{S'}\mathsf P[0\stackrel{L}\longleftrightarrow x\,|\,0\stackrel{L}\longleftrightarrow re_1]\sigma(dx)\\
&\stackrel{\eqref{eq:visibility-general-lower}}\leq 
&D\,\mathsf E[X]\,\int\limits_{S'}\exp\big(-\tfrac1\delta\min(\|x-re_1\|,1)\big)\sigma(dx).\\
\end{eqnarray*}
To bound the integral, we write $S'$ as the disjoint union of sets 
$S_k = \{x\in S'\,:\,(k-1)\delta\leq \|x-re_1\|<k\delta\}$ and dominate the integral over $S_k$ by $\sigma(S_k)e^{-k+1}$, when $k<1+\tfrac1\delta$, and by $\sigma(S_k)e^{-\frac1\delta}$ otherwise. 
Since $\sigma(S_k)\leq c_2(k\delta)^{d-1}$ and $\sum_k\sigma(S_k) = \sigma(S')\leq c_2 s^{d-1}$, the total integral over $S'$ is bounded from above by $c_3(\delta^{d-1} + s^{d-1}e^{-\frac1\delta})$, for some $c_3=c_3(d)$. Thus, 
\[
\mathsf E[X^2] \leq c_3 D(\delta^{d-1}+s^{d-1}e^{-\frac1\delta})\,\mathsf E[X].
\]
By the Paley-Zygmund inequality, 
\[
P_{\mathrm{vis}}(r) \geq \mathsf P[X>0]\geq \frac{\mathsf E[X]^2}{\mathsf E[X^2]}
\geq \tfrac{c_1}{c_3D}\big(\delta^{d-1}+s^{d-1}e^{-\frac1\delta}\big)^{-1}s^{d-1}f(r).
\]
The proof is completed. 
\end{proof}
\begin{remark}\label{rem:visibility-general}
Essentially the same proof gives the following variation of part 2.\ of Proposition~\ref{prop:visibility-general}: if instead of \eqref{eq:visibility-general-lower}, 
\begin{equation}\label{eq:visibility-general-lower-2}
\mathsf P\big[0\stackrel{L}\longleftrightarrow x\,|\,0\stackrel{L}\longleftrightarrow re_1\big]\leq D\exp\big(-\tfrac{1}{\delta}\|x-re_1\|\big),\quad\text{for all $x\in\partial B(r)\cap B(re_1,s)$,}
\end{equation}
then 
\[
P_{\mathrm{vis}}(r)\geq cD^{-1}\big(\tfrac s\delta\big)^{d-1}f(r).
\]
We will use this in the proof of Theorem~\ref{intro:visibility-mainresult} for the Poisson cylinders in $d=2$. 
\end{remark}

\section{Brownian interlacements}\label{sec:BI}

For $\alpha>0$, let $\mathcal I^\alpha$ be a random closed subset of $\R^d$, whose law is characterized by the relations 
\begin{equation}\label{eq:BI-capacity}
\mathsf P \big[\mathcal I^\alpha\cap K = \emptyset\big] = e^{-\alpha\mathrm{cap}(K)},\quad\text{for compact }K\subset\R^d,
\end{equation}
where $\mathrm{cap}(K)$ is the Newtonian capacity of $K$, see \eqref{def:capacity}. 
The \emph{Brownian interlacement at level $\alpha$ with radius $\rho$} is the closed random subset $\mathcal I^\alpha_\rho$ of $\R^d$, defined as 
\[
\mathcal I^\alpha_\rho = B(\mathcal I^\alpha,\rho) = \bigcup\limits_{x\in\mathcal I^\alpha} B(x,\rho).
\]
By \eqref{eq:BI-capacity}, the law of $\mathcal I^\alpha_\rho$, denoted by $\mathsf P^\alpha_\rho$, is characterized by the relations
\begin{equation}\label{eq:BI-capacity-rho}
\mathsf P \big[\mathcal I^\alpha_\rho\cap K = \emptyset\big] = e^{-\alpha\mathrm{cap}(B(K,\rho))},\quad\text{for compact }K\subset\R^d
\end{equation}
(see \cite[(2.32)]{Sznitman-BI}). 

\smallskip

We first introduce capacity and its properties in Section~\ref{sec:BI-capacity} and then prove part (1) of Theorem~\ref{intro:visibility-mainresult} in Sections~\ref{sec:BI-lower} (lower bound of \eqref{intro:visibility-bounds}) and \ref{sec:BI-upper} (upper bound of \eqref{intro:visibility-bounds}).

\subsection{Brownian motion and capacity}\label{sec:BI-capacity}

Let $W$ be a Brownian motion in $\R^d$. We denote by $\mathsf P_x$ the law of $W$ with $W_0=x$ and we write $\mathsf P_\nu$ for $\int_{\R^d}\mathsf P_x[\cdot]\nu(dx)$. For a closed set $K\subset\R^d$, let $H_K=\inf\{t\geq 0\,:\,W_t\in K\}$ be the first entrance time of $W$ in $K$. 
Classically, for any $R_1<R_2$ and $y\in\R^d$ with $R_1<\|y\|<R_2$, 
\begin{equation}\label{eq:BM-hitting}
\mathsf P_y\big[H_{\partial B(R_2)}<H_{\partial B(R_1)}\big] = 
\left\{\begin{array}{ccl}\frac{R_1^{2-d}-\|y\|^{2-d}}{R_1^{2-d}-R_2^{2-d}} &\quad& d=1\text{ or }d\geq 3\\[10pt] 
\frac{\log R_1 - \log \|y\|}{\log R_1 - \log R_2} &\quad& d=2\end{array}\right.
\end{equation}
(see e.g.\ \cite[Theorem~3.18]{MP-BM-book}), in particular, when $d\geq 3$, 
\begin{equation}\label{eq:BM-escape}
\mathsf P_y\big[H_{\partial B(R_1)}=\infty\big] = 1 - \frac{\|y\|^{2-d}}{R_1^{2-d}}\,.
\end{equation}

\smallskip

Let $\sigma_R$ be the uniform distribution on $\partial B(R)$ and define $\mu_R=\tfrac{2\pi^{d/2}R^{d-2}}{\Gamma(d/2-1)}\sigma_R$. The Newtonian capacity of compact set $K$ in $\R^d$ ($d\geq 3$) is defined as 
\begin{equation}\label{def:capacity}
\mathrm{cap}(K) = \mathsf P_{\mu_R}[H_K<\infty], \quad\text{for any $R$, such that $K\subset B(R)$,}
\end{equation}
see e.g.\ \cite[Ch.~3, Th.~1.10]{PortStone}.
Capacity is an invariant under isometries and monotone function on compacts (see \cite[Ch.~3, Prop.~1.12]{PortStone}), and $\mathrm{cap}(K) = \mathrm{cap}(\partial K)$ (see \cite[Ch.~3, Prop.~1.11]{PortStone}). 

\smallskip

By \cite[Ch.~3, Prop.~3.4]{PortStone} (see also \cite[Lemma~3.1]{ET-visibility}), 
there exist constants $c_i=c_i(d,\rho)$, such that for all $r\geq 2$, the capacity of cylinder $\ell_{re_1}(\rho)$ of radius $\rho$ satisfies the bounds
\begin{equation}\label{eq:capacity-cylinder}
c_1\varphi(r)\leq \mathrm{cap}\big(\ell_{re_1}(\rho)\big) \leq c_2\varphi(r),\quad 
\text{with}\,\,\,\varphi(r) = \left\{\begin{array}{ccl} r &\,\,& d\geq 4\\[4pt] \frac{r}{\log r} &\,\,& d=3\end{array}\right.
\end{equation}

\subsection{Proof of the lower bound of \eqref{intro:visibility-bounds}}\label{sec:BI-lower}
Fix $\alpha>0$ and $\rho>0$. Here, we prove that Brownian interlacements $\mathcal I^\alpha_\rho$ satisfies the lower bound of \eqref{intro:visibility-bounds} with $\delta(r)$ as in part (1) of Theorem~\ref{intro:visibility-mainresult}. 
By Proposition~\ref{prop:visibility-general}, it suffices to show that condition \eqref{eq:visibility-general-lower} holds with $\delta=D_1(d,\alpha,\rho)\delta(r)$, $s=2r$ and $D=D_2(d,\alpha,\rho)$. 
By \eqref{eq:BI-capacity-rho}, for $x\in\partial B(r)$, 
\[
\mathsf P^\alpha_\rho[0\stackrel{L}\longleftrightarrow x\,|\,0\stackrel{L}\longleftrightarrow re_1]= 
\exp\Big(-\alpha\big(\mathrm{cap}(\ell_x(\rho)\cup\ell_{re_1}(\rho)) - \mathrm{cap}(\ell_{re_1}(\rho))\big)\Big).
\]
Thus, it suffices to prove that there exist $c=c(d,\rho)$ and $C=C(d,\rho)$ such that for all $r\geq 2$, 
\begin{equation}\label{eq:BI-lowerbound-capacities}
\mathrm{cap}(\ell_x(\rho)\cup\ell_{re_1}(\rho)) - \mathrm{cap}(\ell_{re_1}(\rho))\geq c\tfrac{1}{\delta(r)}\min(\|x-re_1\|,1) - C.
\end{equation}
By the definition of capacity \eqref{def:capacity} and the strong Markov property, for any set $S\subseteq \ell_x(\rho)\setminus \ell_{re_1}(\rho)$ and any $R>r+\rho$, 
\[
\mathrm{cap}(\ell_x(\rho)\cup\ell_{re_1}(\rho)) - \mathrm{cap}(\ell_{re_1}(\rho)) 
\geq \mathsf P_{\mu_R}\big[H_S<H_{\ell_{re_1}(\rho)}]\,\inf\limits_{y\in S}\mathsf P_y[H_{\ell_{re_1}(\rho)}=\infty].
\]
Thus, \eqref{eq:BI-lowerbound-capacities} will follow if we show that for some choice of $S$, for all $r\geq r_0(d,\rho)$ (large enough) and some $c=c(d,\rho)>0$, 
\begin{equation}\label{eq:BI-lowerbound-main}
 \mathsf P_{\mu_R}\big[H_S<H_{\ell_{re_1}(\rho)}]\geq c\varphi(r) \quad\text{and}\quad 
 \inf\limits_{y\in S}\mathsf P_y[H_{\ell_{re_1}(\rho)}=\infty] \geq c \psi(r)\min(\|x-re_1\|,1),
\end{equation}
where $\varphi$ is defined in \eqref{eq:capacity-cylinder} and $\psi(r)= 1$ for $d\geq 4$ and $\frac{1}{\log r}$ for $d=3$. 

\smallskip

For the second part of \eqref{eq:BI-lowerbound-main}, we use the following general lemma. 
\begin{lemma}\label{l:BI-lowerbound-nonhitting}
There exists $c=c(d,\rho)>0$ such that for all $y\in \R^d$, 
\[
\mathsf P_y[H_{\ell_{re_1}(\rho)}=\infty] \geq c \psi(r) \min\big(d(y,\ell_{re_1}(\rho)),1\big).
\]
\end{lemma}
\begin{proof}
Let $y\in \R^d$. The Brownian motion from $y$ avoids $\ell_{re_1}(\rho)$ forever, if 
(a) it exits from the infinite cylinder $B(\R e_1,2\rho)$ before hitting $\ell_{re_1}(\rho)$, after that 
(b) it exits from the infinite cylinder $B(\R e_1,3r)$ before hitting the infinite cylinder $B(\R e_1, \rho)$, and 
(c) it never hits the ball $B(2r)$ after the first exit from the cylinder $B(\R e_1,3r)$. Thus, by the strong Markov property, 
\begin{multline*}
\mathsf P_y[H_{\ell_{re_1}(\rho)}=\infty]\geq\\ 
\mathsf P_y[H_{\partial B(\R e_1, 2\rho)}<H_{\ell_{re_1}(\rho)}]\,
\inf\limits_{z_1} 
\mathsf P_{z_1}[H_{\partial B(\R e_1, 3r)}<H_{\partial B(\R e_1, \rho)}]\,
\inf\limits_{z_2}\mathsf P_{z_2}[H_{B(2r)}=\infty],
\end{multline*}
where the first infimum is over $z_1\in\partial B(\R e_1,2\rho)$ and the second over $z_2\in\partial B(\R e_1, 3r)$. 
By \eqref{eq:BM-hitting} and the fact that the projection of $W$ on any line is a one-dimensional Brownian motion, the first probability is bounded from below by $c\min\big(d(y,\ell_{re_1}(\rho)),1\big)$, for some $c=c(d,\rho)>0$. 
By \eqref{eq:BM-hitting} and the fact that the projection of $W$ on the hyperplane $\{z\in\R^d\,:\,z_1=0\}$ is a $(d-1)$-dimensional Brownian motion, the second probability is bounded from below by $c\psi(r)$, for some $c=c(d,\rho)>0$. Finally, by \eqref{eq:BM-escape}, the third probability is bounded from below by $c=c(d)>0$. The proof is completed. 
\end{proof}

We proceed with the proof of \eqref{eq:BI-lowerbound-main}. Our choice of $S$ for \eqref{eq:BI-lowerbound-main} 
will depend on whether $\|x-re_1\|$ is smaller or bigger than some large constant $\Delta=\Delta(d,\rho)$. 
To determine $\Delta$, let $\beta=\beta(d,\rho)\in (0,1)$ be such that
\[
\mathrm{cap}(\ell_{\beta re_1}(\rho))\leq \frac14 \mathrm{cap}(\ell_{re_1}(\rho)) 
\]
for all $r$ large enough. Given such $\beta$, we fix some $\Delta=\Delta(d,\rho)$ such that for all $x\in \partial B(r)$ with $\|x-re_1\|\geq \Delta$, the distance between the line segment $[\beta x,x]$ and the positive half axis $\R_+e_1$ is at least $1+2\rho$, 
\[
d\big([\beta x,x],\R_+e_1\big)\geq 1+2\rho.
\]

\smallskip

We first prove \eqref{eq:BI-lowerbound-main} for all $x\in\partial B(r)$ with $\|x-re_1\|\geq \Delta$. In this case, we choose $S\subset \ell_x(\rho)\setminus \ell_{re_1}(\rho)$ as 
\[
S = B([\beta x,x],\rho) = \bigcup\limits_{y\in [\beta x,x]} B(y,\rho). 
\]
By the definition of $\Delta$, 
\[
d(S,\ell_{re_1}(\rho)) \geq d\big([\beta x,x],\R_+e_1\big) - 2\rho \geq 1. 
\]
Thus, by Lemma~\ref{l:BI-lowerbound-nonhitting}, 
\[
\inf\limits_{y\in S}\mathsf P_y[H_{\ell_{re_1}(\rho)}=\infty] \geq c \psi(r)
\geq c\psi(r)\min\big(\|x-re_1\|,1\big),
\]
which verifies the second inequality in \eqref{eq:BI-lowerbound-main}. 
Let $S' = B([\beta re_1, re_1],\rho)$. We have 
\begin{multline*}
\mathsf P_{\mu_R}\big[H_S<H_{\ell_{re_1}(\rho)}] 
\geq \mathsf P_{\mu_R}\big[H_S= H_{\ell_x(\rho)\cup \ell_{re_1}(\rho)}<\infty]\\
= \frac12\, \mathsf P_{\mu_R}\big[H_{S\cup S'}= H_{\ell_x(\rho)\cup \ell_{re_1}(\rho)}<\infty]\\
\geq \frac12\,\mathsf P_{\mu_R}\big[H_{\ell_x(\rho)\cup \ell_{re_1}(\rho)}<\infty] - 
\frac12\big(\mathsf P_{\mu_R}\big[H_{\ell_{\beta re_1}(\rho)}<\infty] 
+ \mathsf P_{\mu_R}\big[H_{\ell_{\beta x}(\rho)}<\infty]\big)\\
\stackrel{\eqref{def:capacity}}=\frac12\,\mathrm{cap}\big(\ell_x(\rho)\cup \ell_{re_1}(\rho)\big) - \mathrm{cap}\big(\ell_{\beta re_1}(\rho)\big) \geq \frac14\, \mathrm{cap}\big(\ell_{re_1}(\rho)\big),
\end{multline*}
where the last inequality follows from the monotonicity of the capacity and the choice of $\beta$. Now, the first inequality in \eqref{eq:BI-lowerbound-main} follows immediately from \eqref{eq:capacity-cylinder}. 

\smallskip

Next, we prove \eqref{eq:BI-lowerbound-main} for all $x\in\partial B(r)$ with $\|x-re_1\|\leq \Delta$. 
Without loss of generality, we may assume that 
\[
x = \big(r\cos\varphi, r\sin\varphi, 0,\ldots, 0\big),\quad\text{for some }\varphi\in(0,\frac\pi4). 
\]
(In fact, $\|x-re_1\| = 2r\sin\frac\varphi 2 \leq \Delta$ implies that $\varphi \leq 4\sin\frac\varphi2 \leq \frac{2\Delta}{r}$.)

\begin{figure}[!tp]
\centering
\resizebox{15cm}{!}{\includegraphics{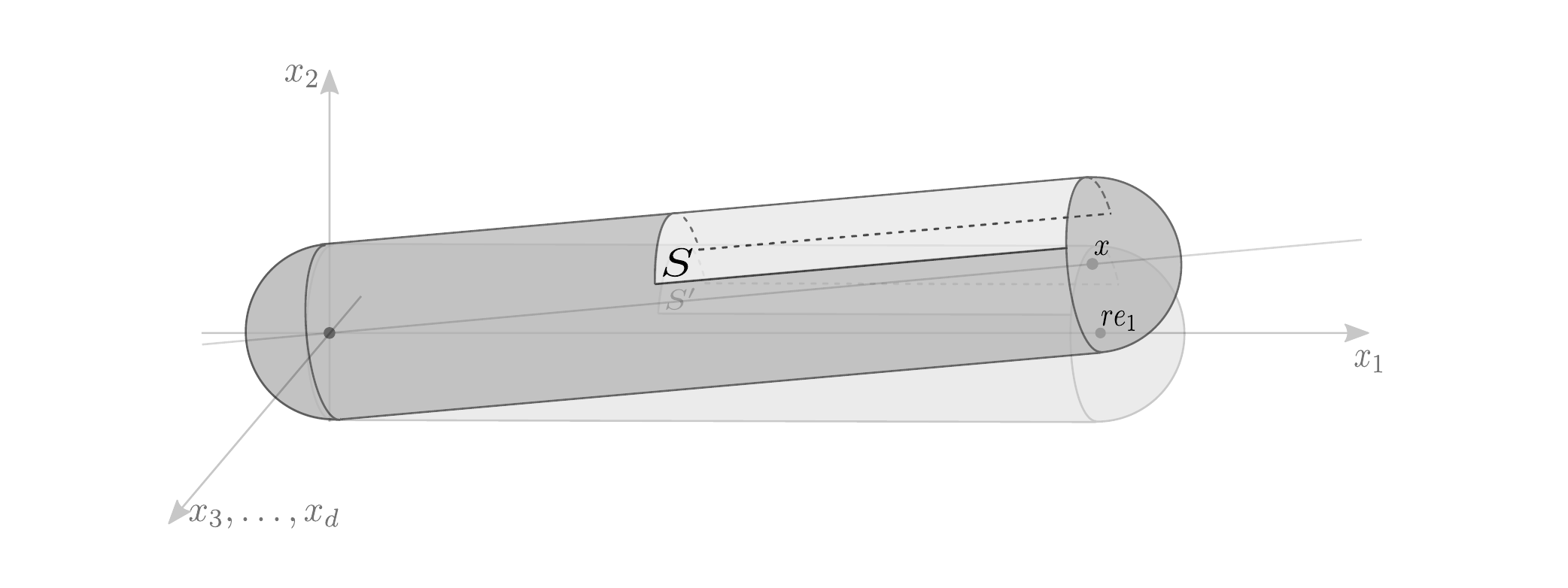}}
\caption{$S$ is the rotation of the set $S' = \{y\in \partial \ell_{re_1}(\rho)\,:\,\frac12r\leq y_1\leq r, y_2\geq \frac12\rho\}$ by angle $\varphi$ parallel to the coordinate $x_1x_2$-plane.}
\label{fig:S}
\end{figure}

Let $S' = \{y\in \partial \ell_{re_1}(\rho)\,:\,\frac12r\leq y_1\leq r, y_2\geq \frac12\rho\}$ and define $S$ as the rotation of $S'$ by angle $\varphi$ parallel to the coordinate plane $\R^2\times\{0\}^{d-2}$:
\[
S = \big\{\big(y_1\cos\varphi - y_2\sin\varphi, y_1\sin\varphi + y_2\cos\varphi, y_3, \ldots, y_d\big)\in\R^d\,:\, y\in S'\big\},
\]
see Figure~\ref{fig:S}. Notice that $S\subseteq \partial \ell_x(\rho)$. Furthermore, $d(S,\ell_{re_1}(\rho)) \geq c\|x-re_1\|$, for some $c=c(d)$ and all $r\geq r_0(d,\rho)$. Indeed, it suffices to show that $d(S,\R e_1)\geq \rho + c\|x-re_1\|$. Since $\|x-re_1\| = 2r\sin\frac\varphi2$, it suffices to find a $c=c(d)$ such that 
\[
(y_1\sin\varphi + y_2\cos\varphi)^2 + \sum\limits_{i=3}^d y_i^2 \geq (\rho + 2rc\sin\frac\varphi2)^2,\quad\text{for all }y\in S'.
\]
Using the fact that $\sum_{i=2}^d y_i^2 = \rho^2$, the inequality follows if 
\[
(y_1^2 - y_2^2)\sin^2\varphi \geq 4c^2r^2\sin^2\frac\varphi2\quad \text{and} \quad 
y_1y_2\sin\varphi\cos\varphi \geq 2\rho r c \sin\frac\varphi2. 
\]
Since $y_1\geq \frac12r$, $\frac12\rho\leq y_2\leq \rho$ and $\cos\varphi\geq \frac1{\sqrt2}$, the above two inequalities hold for some $c=c(d)$ if $r\geq 4\rho$. 
We conclude from $d(S,\ell_{re_1}(\rho)) \geq c\|x-re_1\|$ and Lemma~\ref{l:BI-lowerbound-nonhitting} that $S$ satisfies the second inequality in \eqref{eq:BI-lowerbound-main}.

\smallskip

It remains to show that $S$ also satisfies the first inequality in \eqref{eq:BI-lowerbound-main}. 
Note that $\ell_{re_1}(\rho)\cup\ell_x(\rho)\subset \ell_{re_1}(2\Delta)$. Consider $\overline S = \{y\in\partial \ell_{re_1}(2\Delta)\,:\,y_1\geq \frac12r\}$. Notice that  
\[
\inf\limits_{y\in \overline S}P_y\big[H_S<H_{\ell_{re_1}(\rho)}\big]\geq c=c(d,\rho)>0.
\]
Indeed, for every $y\in \overline S$, there exists a polygonal path $\gamma$ from $y$ to some $z=z(y)\in\overline S$ with $\frac12r + \Delta\leq z_1 \leq r - \Delta$, $z_2 = 2\Delta$ and $z_3=\ldots=z_d=0$, which consists of at most $2d$ parallel to coordinate axes line segments of length at most $4 \Delta$ each, and so that $\gamma\cap\ell_{re_1}(2\Delta) = \{y,z\}$. Let $z' = (z_1,0,\ldots,0)$ and consider the corridor $T=T(y) = B(\gamma,\frac12\rho)\cup B([z,z'],\frac12\rho)$ from $y$ to $z'$ of radius $\frac12\rho$. Notice that any continuous curve in $T$ from $y$ to $z'$ intersects $S$ before $\ell_{re_1}(\rho)$. Furthermore, the probability that a Brownian motion started in $y$ hits $B(z',\frac12\rho)$ before leaving $T$ is bigger than $c=c(d,\rho)>0$ for all $y\in\overline S$, regardless of the choice of $\gamma$ as above. Hence, $P_y\big[H_S<H_{\ell_{re_1}(\rho)}\big]\geq c$. 

Next, by symmetry, 
\[
\mathsf P_{\mu_R}\big[H_{\overline S}=H_{\ell_{re_1}(2\Delta)}<\infty]
= \frac12 \mathsf P_{\mu_R}\big[H_{\ell_{re_1}(2\Delta)}<\infty]
= \frac12\mathrm{cap}\big(\ell_{re_1}(2\Delta)\big) \stackrel{\eqref{eq:capacity-cylinder}}\geq c\varphi(r),
\]
for some $c=c(d,\rho)>0$. Putting the two bounds together, we obtain that 
\[
\mathsf P_{\mu_R}\big[H_S<H_{\ell_{re_1}(\rho)}\big]\geq \mathsf P_{\mu_R}\big[H_{\overline S}=H_{\ell_{re_1}(2\Delta)}<\infty]\,\inf\limits_{y\in \overline S}P_y\big[H_S<H_{\ell_{re_1}(\rho)}\big]\geq c\varphi(r),
\]
for some $c=c(d,\rho)>0$, which is the first inequality in \eqref{eq:BI-lowerbound-main}. 

\smallskip

The proof of \eqref{eq:BI-lowerbound-main} is completed for all $x\in\partial B(r)$. \qed

\subsection{Proof of the upper bound of \eqref{intro:visibility-bounds}}\label{sec:BI-upper}

Fix $\alpha>0$ and $\rho>0$. 
Here, we prove that Brownian interlacements $\mathcal I^\alpha_\rho$ satisfies the upper bound of \eqref{intro:visibility-bounds} with $\delta(r)$ as in part (1) of Theorem~\ref{intro:visibility-mainresult}. 
By Proposition~\ref{prop:visibility-general}, it suffices to show that condition \eqref{eq:visibility-general-upper} holds with $\delta=D_1(d,\rho)\delta(r)$ and $D=D_2(d,\alpha,\rho)$. We choose $D_1$ such that $\delta<\rho$ for all $r\geq 2$. 

If $\partial B(r)\cap B(re_1,\delta)$ is visible from $0$ through the vacant set of $\mathcal I^\alpha_\rho$, then the line segment $\ell_{re_1}$ is visible through the vacant set of $\mathcal I^\alpha_{\rho-\delta}$. Thus, it suffices to prove that there exists $D=D(d,\alpha,\rho)$ such that for all $r\geq 2$, 
\[
\mathsf P^\alpha_{\rho-\delta}[0\stackrel{L}\longleftrightarrow re_1]\leq D\mathsf P^\alpha_\rho[0\stackrel{L}\longleftrightarrow re_1].
\]
By \eqref{eq:BI-capacity-rho}, this is equivalent to 
\begin{equation}\label{eq:visibility-BI-upper-main}
\mathrm{cap}(\ell_{re_1}(\rho))- \mathrm{cap}(\ell_{re_1}(\rho-\delta)) \leq C,
\end{equation}
for some $C=C(d,\rho)$ and $r\geq r_0(d,\rho)$ (large enough). 
By the definition of capacity \eqref{def:capacity}, for any $R>r+\rho$,  
\begin{equation}\label{eq:visibility-BI-upper-integral}
\mathrm{cap}(\ell_{re_1}(\rho))- \mathrm{cap}(\ell_{re_1}(\rho-\delta)) 
= \mathsf P_{\mu_R}\big[H_{\ell_{re_1}(\rho)}<\infty,\,H_{\ell_{re_1}(\rho-\delta)}=\infty\big].
\end{equation}

\smallskip

If $d\geq 4$, then by the strong Markov property and \eqref{def:capacity},
the right hand side in \eqref{eq:visibility-BI-upper-integral} is bounded from above by 
\[
\mathrm{cap}(\ell_{re_1}(\rho))\,\sup\limits_{y\in\partial \ell_{re_1}(\rho)}\mathsf P_y[H_{\ell_{re_1}(\rho-\delta)}=\infty]
\stackrel{\eqref{eq:BM-escape},\eqref{eq:capacity-cylinder}}\leq 
Cr\delta = CD_1,
\]
for some $C=C(d,\rho)$, where we estimated the probability to avoid the cylinder $\ell_{re_1}(\rho-\delta)$ by the probability to avoid a ball of radius $(\rho-\delta)$ at distance $\delta$ from $y$ contained in $\ell_{re_1}(\rho-\delta)$. Hence \eqref{eq:visibility-BI-upper-main} holds. 

\smallskip

It remains to prove \eqref{eq:visibility-BI-upper-main} for $d=3$. 
In this case, the probability that Brownian motion from $y\in\partial \ell_{re_1}(\rho)$ avoids $\ell_{re_1}(\rho-\delta)$ scales differently for different $y$'s, depending on how close $y$ is to one of the ends of the cylinder. 
Define the line segments 
\[
\ell_l=[0,r^{\frac23}e_1],\quad 
\ell_c=[r^{\frac23}e_1,re_1-r^{\frac23}e_1],
\quad
\ell_r=[re_1-r^{\frac23}e_1,re_1],
\]
and let $\ell_j(\rho) = B(\ell_j,\rho)$, for $j\in\{l,c,r\}$, be the respective cylinders of radius $\rho$. 
By the strong Markov property and \eqref{def:capacity}, the integral in \eqref{eq:visibility-BI-upper-integral} is bounded from above by 
\begin{equation}\label{eq:visibility-BI-upper-lcr}
\sum\limits_{j\in\{l,c,r\}}\mathrm{cap}(\ell_j(\rho))\,\sup\limits_{y\in\partial \ell_j(\rho)}\mathsf P_y[H_{\ell_{re_1}(\rho-\delta)}=\infty].
\end{equation}
In the remainder of the proof, we denote by $c$ and $C$ positive constants that only depend on $\rho$ and whose value may change from place to place. 
By \eqref{eq:capacity-cylinder}, $\mathrm{cap}(\ell_l(\rho))=\mathrm{cap}(\ell_r(\rho))\leq Cr^{\frac23}$ and $\mathrm{cap}(\ell_c(\rho))\leq C\frac{r}{\log r}$.

When $j\in\{l,r\}$, we estimate the probability to avoid the cylinder $\ell_{re_1}(\rho-\delta)$ by the probability to avoid a ball of radius $(\rho-\delta)$ at distance $\delta$ from $y$, just as in the previous case of $d\geq 4$. Thus, by \eqref{eq:BM-escape}, 
\[
\sup\limits_{y\in\partial \ell_l(\rho)}\mathsf P_y[H_{\ell_{re_1}(\rho-\delta)}=\infty]
=
\sup\limits_{y\in\partial \ell_r(\rho)}\mathsf P_y[H_{\ell_{re_1}(\rho-\delta)}=\infty]\leq C\delta.
\]
Now, if a Brownian motion started in $y\in\partial \ell_c(\rho)$ aviods $\ell_{re_1}(\rho-\delta)$, then either (a) it exits from the doubly infinite cylinder $B(\R e_1,r^{\frac13})$ before hitting the infinite cylinder $B(\R e_1,\rho-\delta)$ or (b) it exits from the ball $B(y', r^{\frac23})$ before it exits from the infinite cylinder $B(\R e_1,r^{\frac13})$; here $y'$ is the closest point of $\ell_c$ to $y$. The first possibility has probability at most $C\frac{\delta}{\log r}$ by \eqref{eq:BM-hitting} and the fact that the projection of Brownian motion on the hyperplane $\{z\in\R^3\,:\,z_1=0\}$ is a $2$-dimensional Brownian motion. The second possibility has probability at most $(1-c)^{\lfloor\frac14r^{\frac13}\rfloor-1}$, since for any $z\in B(\R e_1,r^{\frac13})$, the Brownian motion from $z$ exits from the cylinder $B(\R e_1,r^{\frac13})$ before leaving the ball $B(z,4r^{\frac13})$ with probability $\geq c$ uniformly in $z$ and $r$. Thus, if it exits from the ball $B(y', r^{\frac23})$ before leaving the cylinder $B(\R e_1,r^{\frac13})$,
then the above event should not occur at least $\lfloor\frac14r^{\frac13}\rfloor-1$ times at the first exit times from the balls $B(y',4r^{\frac13}k)$, $1\leq k\leq \lfloor\frac14r^{\frac13}\rfloor-1$. Hence 
\[
\sup\limits_{y\in\partial \ell_c(\rho)}\mathsf P_y[H_{\ell_{re_1}(\rho-\delta)}=\infty]\leq C\tfrac{\delta}{\log r} + (1-c)^{\lfloor\frac14r^{\frac13}\rfloor-1}\,.
\]
Recall that $\delta=D_1\delta(r) = D_1\frac{\log^2r}{r}$. Thus, the sum in \eqref{eq:visibility-BI-upper-lcr} is bounded from above by $C=C(\rho)$, which proves \eqref{eq:visibility-BI-upper-main}. \qed

\section{Poisson cylinders}\label{sec:PC}

Let $SO_d$ be the topological group of rigid rotations of $\R^d$. We denote by $\nu$ the unique Haar measure on $SO_d$ such that $\nu(SO_d)=1$. For $K\subset\R^d$ and $\phi\in SO_d$, we denote by $K_\phi$ the rotation of $K$ by $\phi$. 

Let $H$ be the hyperplane $\{x=(x_1,\ldots, x_d)\in\R^d\,:\,x_1=0\}$ and $\pi$ the orthogonal projection on $H$. Denote by $\lambda_k$ the $k$-dimensional Lebesgue measure. 

\smallskip

For $\alpha>0$, let $\mathcal L^\alpha$ be a random closed subset of $\R^d$ ($d\geq 2$), whose law is characterized by the relations 
\begin{equation}\label{eq:PC-mu}
\mathsf P \big[\mathcal L^\alpha\cap K = \emptyset\big] = e^{-\alpha\mu(K)},\quad\text{for compact }K\subset\R^d,
\end{equation}
where 
\[
\mu(K) = \int\limits_{SO_d}\lambda_{d-1}\big(\pi(K_\phi)\big)\,\nu(d\phi)
\]
(see e.g.\ \cite[(2.8)]{TW-cylinders}). 
The \emph{Poisson cylinders model at level $\alpha$ with radius $\rho$} is the closed random subset $\mathcal L^\alpha_\rho$ of $\R^d$, defined as 
\[
\mathcal L^\alpha_\rho = B(\mathcal L^\alpha,\rho) = \bigcup\limits_{x\in\mathcal L^\alpha} B(x,\rho).
\]
By \eqref{eq:PC-mu}, the law of $\mathcal L^\alpha_\rho$, denoted by $\mathsf P^\alpha_\rho$, is characterized by the relations
\begin{equation}\label{eq:PC-mu-rho}
\mathsf P \big[\mathcal L^\alpha_\rho\cap K = \emptyset\big] = e^{-\alpha\mu(B(K,\rho))},\quad\text{for compact }K\subset\R^d.
\end{equation}

\smallskip

In Section~\ref{sec:PC-auxiliary}, we collect some geometric properties of the volume of cylinders and of orthogonal projections. We prove part (2) of Theorem~\ref{intro:visibility-mainresult} in Sections~\ref{sec:PC-upper} (upper bound of \eqref{intro:visibility-bounds}) and \ref{sec:PC-lower} (lower bound of \eqref{intro:visibility-bounds}).

\subsection{Auxiliary results}\label{sec:PC-auxiliary}

Recall that $\ell_x(\rho) = B(\ell_x,\rho)$ is the closed $\rho$-neighborhood of the line segment $\ell_x=[0,x]$, which we call ($d$-dimensional) cylinder around $[0,x]$. 
In Lemma~\ref{l:mu-segment-PC} we compute the volume of $\ell_x(\rho)$; in Lemma~\ref{l:volume-symmetric-difference}, we provide a bound on the volume of the symmetric difference of two cylinders $\ell_x(\rho)\Delta\ell_y(\rho)$; in Lemma~\ref{l:vector-projections}, we show that any two vectors in $\R^d$ can be orthogonally projected onto one of the $d$ coordinate hyperplanes, so that their norms and the angle between them are not too small compared to the original ones.

\begin{lemma}\label{l:mu-segment-PC}
Let $r,\rho>0$. 
\begin{itemize}
\item
If $d=1$, then $\lambda_d(\ell_x(\rho)) = |x|+2\rho$. 
\item
If $d\geq 2$, then $\lambda_d(\ell_x(\rho)) = \kappa_{d-1}\rho^{d-1}\|x\| + \kappa_d\rho^d$,
where $\kappa_s$ is the volume of the $s$-dimensional Euclidean unit ball.
\end{itemize}
\end{lemma}
\begin{proof}
Immediate from the definition of $\ell_x(\rho)$. 
\end{proof}

\begin{lemma}\label{l:volume-symmetric-difference}
Let $\rho>0$. 
\begin{itemize}
\item
If $d=1$, then for all $x,y\in \R^d$, $\lambda_d\big(\ell_x(\rho)\Delta\ell_y(\rho)\big) = |x-y|$.
\item
If $d\geq 2$, then there exists $c_3=c_3(d)>0$, such that for all $x,y\in \R^d$ with $r = \min(\|x\|,\|y\|)>8\rho$, 
\[
\lambda_d\big(\ell_x(\rho)\Delta \ell_y(\rho)\big) \geq c_3 r \rho^{d-2} \min\big(r\sin\frac\varphi2,\rho\big),
\]
where $\varphi\in[0,\pi]$ is the angle between $x$ and $y$.
\end{itemize}
\end{lemma}
\begin{proof}
The statement for $d=1$ is immediate by considering separately the cases $x,y>0$ and $x<0<y$. 

\smallskip

Let $d\geq 2$. It suffices to prove that for all $x,y\in\R^d$ with $\|x\|=\|y\|=r>8\rho$, 
\begin{equation}\label{eq:volume-equal-norms}
\lambda_d\big(\ell_y(\rho)\setminus \ell_x(\rho)\big) \geq c_3 r \rho^{d-2} \min\big(r\sin\frac\varphi2,\rho\big).
\end{equation}
Indeed, if $x,y\in\R^d$ are abritrary with $\|x\|\leq \|y\|$, then 
\[
\lambda_d\big(\ell_x(\rho)\Delta \ell_y(\rho)\big)\geq \lambda_d\big(\ell_y(\rho)\setminus \ell_x(\rho)\big)\geq 
\lambda_d\big(\ell_{\widetilde y}(\rho)\setminus \ell_x(\rho)\big),
\]
where $\widetilde y = \frac{\|x\|}{\|y\|}y$, and the desired inequality follows from \eqref{eq:volume-equal-norms} applied to $x$ and $\widetilde y$. 

\smallskip

To prove \eqref{eq:volume-equal-norms}, we may assume, without loss of generality, that $x=(r,0,\ldots,0)$ and $y=(r\cos\varphi, r\sin\varphi, 0,\ldots, 0)$, for some $\varphi\in[0,\pi]$. 
If $r\sin\frac\varphi2> 4\rho$, then $d(\frac12y,\R_+e_1)\geq \frac12 r\sin\frac\varphi2> 2\rho$. Hence, in this case, the $\rho$-neighborhood of the line segment $[\frac12 y,y]$ is disjoint from $\ell_x(\rho)$, which implies that 
\[
\lambda_d\big(\ell_y(\rho)\setminus \ell_x(\rho)\big) \geq \lambda_d\big(\ell_{\frac12y}(\rho)\big) 
\stackrel{(\mathrm{L. }\ref{l:mu-segment-PC})}=\frac12\kappa_{d-1}\rho^{d-1} r + \kappa_d \rho^d
\geq \frac12\kappa_{d-1} r \rho^{d-2} \min\big(r\sin\frac\varphi2,\rho\big).
\]
\begin{figure}[!tp]
\centering
\resizebox{15cm}{!}{\includegraphics{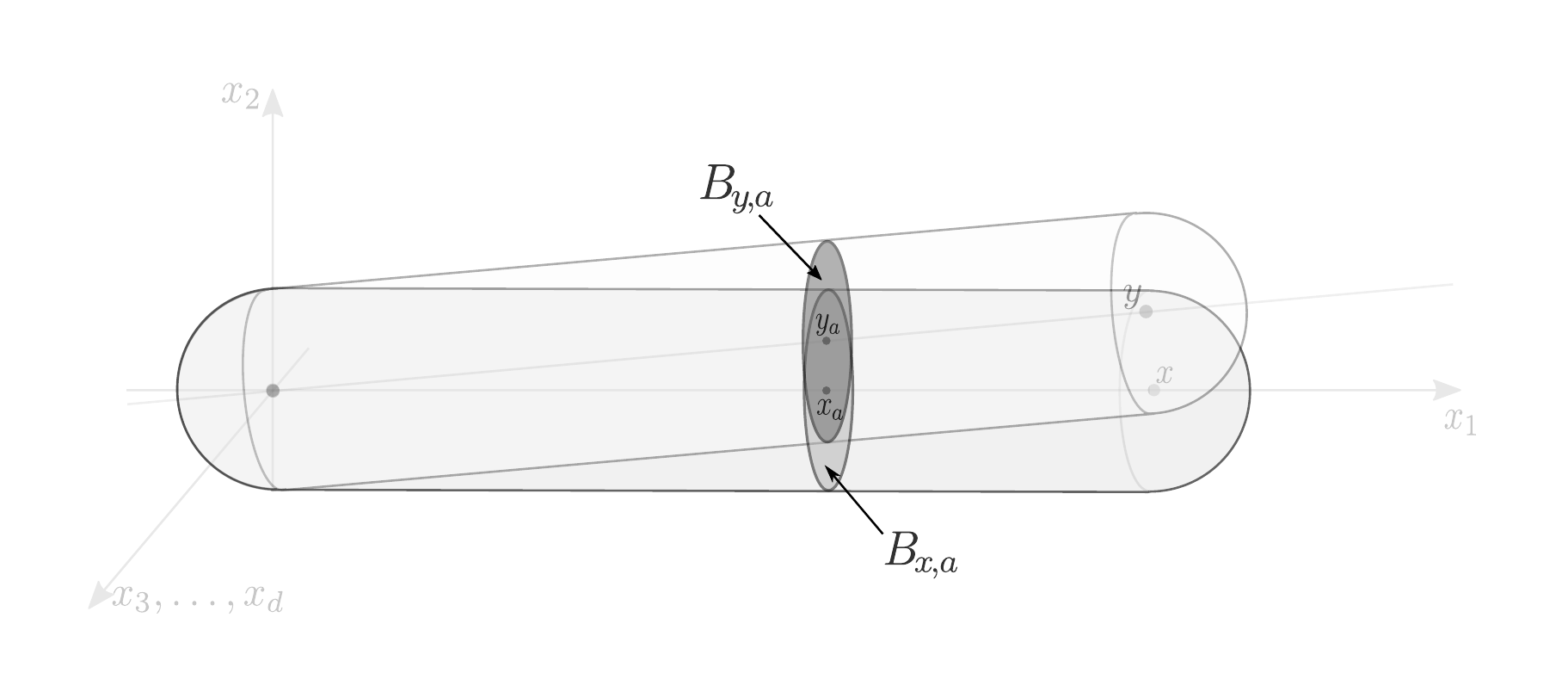}}
\caption{For $x=(r,0,\ldots, 0)$, $y=(r\cos\varphi,r\sin\varphi,0,\ldots, 0)$ and $a\in[0,r\cos\varphi]$, cylinder $\ell_x(\rho)$ intersects hyperplane $H_a = \{z\in\R^d:z_1=a\}$ on $(d-1)$-dimensional ball $B_{x,a}$ of radius $\rho$ and center at $x_a=(a,0,\ldots,0)$, cylinder $\ell_y(\rho)$ intersects $H_a$ on a subset (generally ellipsoid), which contains $(d-1)$-dimensional ball $B_{y,a}$ of radius $\rho$ and center $y = (a,a\tan\varphi,0,\ldots,0)$.}
\label{fig:Ha}
\end{figure}
It remains to consider the case $r\sin\frac\varphi2\leq 4\rho$. Since $r>8\rho$, we have $\sin\frac\varphi2\leq \frac12$, which implies that $\varphi\leq \frac\pi3$. 
Note that the cylinder $\ell_y(\rho)$ intersects any hyperplane $H_a = \{z\in\R^d\,:\,z_1=a\}$, $a\in[0,r\cos\varphi]$, on a subset (generally ellipsoid), which contains the $(d-1)$-dimensional ball $B_{y,a}$ in $H_a$ of radius $\rho$ and center at $y_a = (a,a\tan\varphi,0,\ldots, 0)\in\ell_y$, and the cylinder $\ell_x(\rho)$ intersects $H_a$ on the $(d-1)$-dimensional ball $B_{x,a}$ in $H_a$ of radius $\rho$ and center at $x_a=(a,0,\ldots, 0)$, see Figure~\ref{fig:Ha}. Thus, for each $a\in[0,r\cos\varphi]$, 
\[
\lambda_{d-1}\big((\ell_y(\rho)\setminus \ell_x(\rho))\cap H_a\big) \geq \lambda_{d-1}\big(B_{y,a}\setminus B_{x,a}\big).
\]
Let $\delta = a\tan\varphi$ be the distance between the centers of the balls and define $\theta=\frac\pi4$ (any other choice of $\theta\in (0,\frac\pi2)$ would also do). Let $S_a$ be the subset of the spherical cone of the ball $B_{y,a}$ in the direction of the second coordinate axis and with opening angle $\theta$, whose points are at distance at least $\rho-\delta\cos\theta$ from the center of the ball:
\[
S_a = \big\{z\in B_{y,a}\,:\,z_2 \geq \|z-y_a\|\cos\theta\text{ and }\|z-y_a\|\geq \rho-\delta\cos\theta\big\},
\]
see Figure~\ref{fig:Sa}. 
\begin{figure}[!tp]
\centering
\resizebox{15cm}{!}{\includegraphics{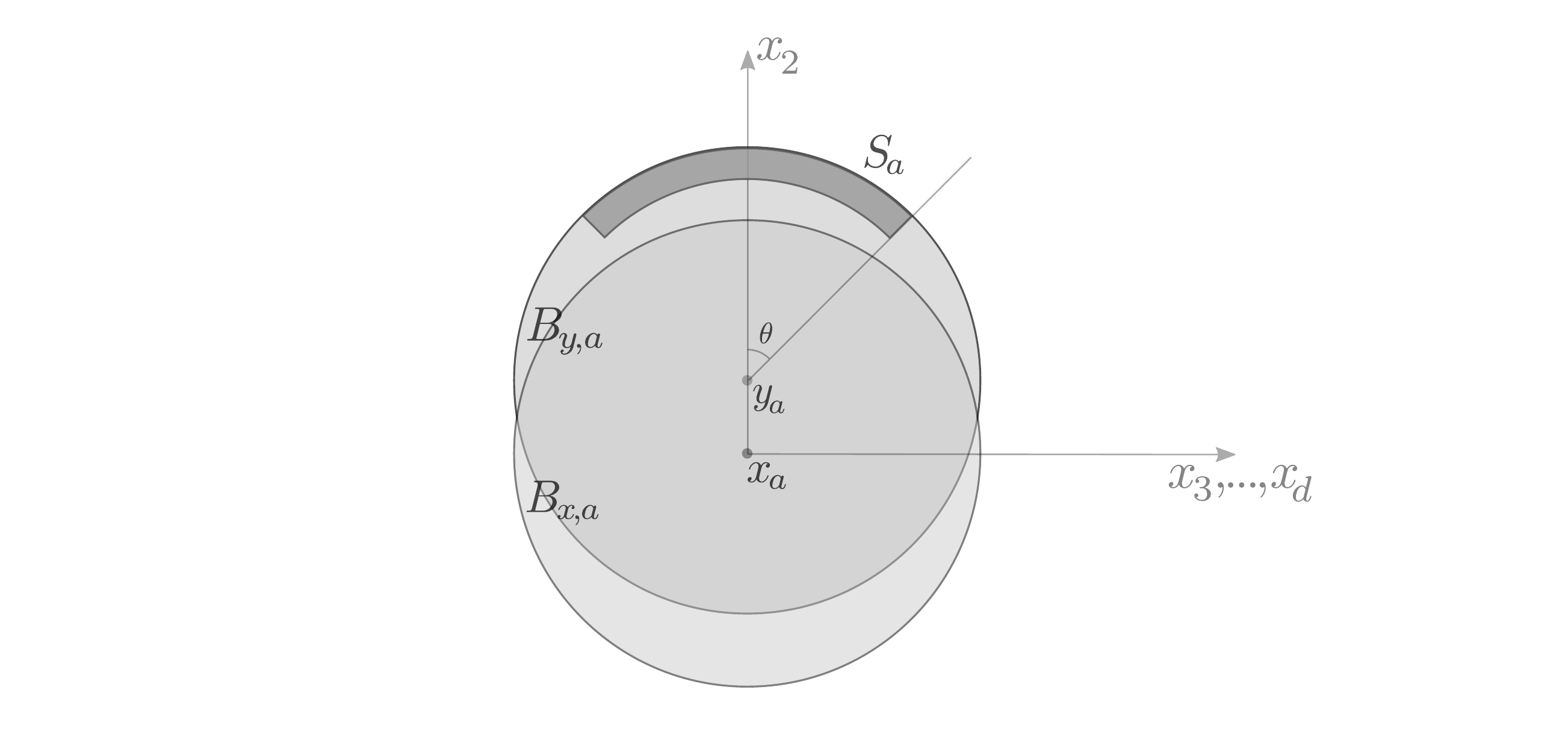}}
\caption{Set $S_a = \big\{z\in B_{y,a}\,:\,z_2 \geq \|z-y_a\|\cos\theta\text{ and }\|z-y_a\|\geq \rho-\delta\cos\theta\big\}$.}
\label{fig:Sa}
\end{figure}
Note that $S_a$ is disjoint from the ball $B_{x,a}$. Indeed, if $\delta\geq\rho$, then it is trivial, if $\delta<\rho$, then for any $z\in S_a$, by the law of cosines applied to the triangle $x_a$, $y_a$ and $z$, we obtain that 
\begin{eqnarray*}
\|z-x_a\|^2 &\geq &\|z-y_a\|^2 + \|y_a-x_a\|^2 + 2\|z-y_a\|\,\|y_a-x_a\|\cos\theta\\
&\geq &\big(\|z-y_a\| + \|y_a-x_a\|\cos\theta\big)^2>\rho^2.
\end{eqnarray*}
Thus, $\lambda_{d-1}\big(B_{y,a}\setminus B_{x,a}\big)\geq \lambda_{d-1}(S_a)$. If $\delta\geq\rho$, then $\lambda_{d-1}(S_a)\geq c\rho^{d-1}$, for some $c=c(d)>0$. If $\delta<\rho$, then  
\[
\lambda_{d-1}(S_a)\geq c\big(\rho^{d-1} - (\rho - \delta\cos\theta)^{d-1}\big)\geq c \rho^{d-2}\delta\cos\theta,
\]
for some $c=c(d)$. Note that $\delta = a \tan\varphi \geq a\sin\varphi\geq a\sin\frac\varphi2$. Thus, for any $\delta>0$, $\lambda_{d-1}(S_a)\geq c\frac ar \rho^{d-2}\min\big(r\sin\frac\varphi2,\rho\big)$, for some $c=c(d)$. Integration over $a$ now gives
\[
\lambda_d\big(\ell_y(\rho)\setminus \ell_x(\rho)\big)
\geq \int\limits_0^{r\cos\varphi}\lambda_{d-1}(S_a)\,da \geq \frac18 c r\rho^{d-2}\min\big(r\sin\frac\varphi2,\rho\big).
\]
The proof of the lemma is completed. 
\end{proof}

\begin{lemma}\label{l:vector-projections}
Let $d\geq 3$. Let $H_i$ be the hyperplane $\{z=(z_1,\ldots, z_d)\in\R^d\,:\,z_i=0\}$ and $\pi_i$ the orthogonal projection on $H_i$. For any vectors $x,y\in\R^d$ at angle $\varphi\in[0,\pi]$, there exists $i$ such that $\|\pi_i(x)\|\geq \frac1{\sqrt d}\|x\|$, $\|\pi_i(y)\|\geq \frac1{\sqrt d}\|y\|$ and $\sin\varphi_i\geq \frac1{\sqrt d}\sin\varphi$, where $\varphi_i\in[0,\pi]$ is the angle between the vectors $\pi_i(x)$ and $\pi_i(y)$. 
\end{lemma}
\begin{proof}
It suffices to prove that for any plane $\Pi$ passing through $0$ there exists $i$ such that the statement of the lemma holds for all linearly independent $x,y\in\Pi$. Let $x,y\in\Pi$ be linearly independent and choose an orthonormal basis $f^{(3)},\ldots, f^{(d)}$ of the orthogonal complement of $\Pi$. There exists $i$ such that the $i$-th coordinate of $f^{(d)}$ satisfies $|f^{(d)}_i| \geq \frac{1}{\sqrt d}$. We prove that the statement holds for any such choice of $i$. Without loss of generality, we may assume that $i=d$ and $f^{(d)}_d=\cos\theta>0$, where $\theta$ be the angle between $f^{(d)}$ and the $d$-th coordinate axis. 

The $(d-1)$-dimensional volume of the parallelepiped $P$ spanned by the vectors $x$, $y$, $f^{(3)},\ldots, f^{(d-1)}$ equals to $\|x\|\|y\|\sin\varphi$ and the $(d-1)$-dimensional volume of its projection $\pi_d(P)$ on $H_d$ equals to 
$\|x\|\|y\|\sin\varphi\cos\theta$. On the other hand, $\pi_d(P)$ is the parallelepiped spanned by the vectors $\pi_d(x)$, $\pi_d(y)$, $\pi_d(f^{(3)}),\ldots, \pi_d(f^{(d-1)})$. Since $\|\pi_d(f^{(j)})\|\leq \|f^{(j)}\|=1$ for all $j$, it follows (for example, by the Gram-Schmidt orthogonalization) that the $(d-1)$-dimensional volume of $\pi_d(P)$ is bounded from above by 
\[
\big(\|\pi_d(x)\|\|\pi_d(y)\|\sin\varphi_d\big)\,\|\pi_d(f^{(3)})\|\ldots\|\pi_d(f^{(d-1)})\| \leq \|\pi_d(x)\|\|\pi_d(y)\|\sin\varphi_d.
\]
Thus, for any linearly independent $x,y\in \Pi$, $\|\pi_d(x)\|\|\pi_d(y)\|\sin\varphi_d\geq \|x\|\|y\|\sin\varphi\cos\theta=\frac1{\sqrt d}\|x\|\|y\|\sin\varphi$. Since $\|\pi_d(x)\|\leq \|x\|$ and $\|\pi_d(y)\|\leq \|y\|$, we conclude that $\sin\varphi_d\geq \frac{1}{\sqrt d}\sin\varphi$. Further, by choosing $y$ orthogonal to $x$ and with the same norm, it follows that for any $x\in\Pi$, $\|\pi_d(x)\|\geq \frac1{\sqrt d}\|x\|$. 
The proof is completed. 
\end{proof}

\subsection{Proof of the upper bound of \eqref{intro:visibility-bounds}}\label{sec:PC-upper}

Fix $\alpha>0$ and $\rho>0$. 
Here, we prove that Poisson cylinders model $\mathcal L^\alpha_\rho$ satisfies the upper bound of \eqref{intro:visibility-bounds} with $\delta(r)$ as in part (2) of Theorem~\ref{intro:visibility-mainresult}. 
By Proposition~\ref{prop:visibility-general}, it suffices to show that condition \eqref{eq:visibility-general-upper} holds with $\delta=\frac12\rho\,\delta(r)$ and $D=D(d,\alpha,\rho)$. 
Note that $\delta<\rho$ for all $r\geq 1$.

If $\partial B(r)\cap B(re_1,\delta)$ is visible from $0$ through the vacant set of $\mathcal L^\alpha_\rho$, then the line segment $\ell_{re_1}$ is visible through the vacant set of $\mathcal L^\alpha_{\rho-\delta}$. Thus, it suffices to prove that there exists $D=D(d,\alpha,\rho)$ such that for all $r\geq 1$, 
\[
\mathsf P^\alpha_{\rho-\delta}[0\stackrel{L}\longleftrightarrow re_1]\leq D\mathsf P^\alpha_\rho[0\stackrel{L}\longleftrightarrow re_1].
\]
By \eqref{eq:PC-mu-rho}, this is equivalent to 
\[
\mu(\ell_{re_1}(\rho))- \mu(\ell_{re_1}(\rho-\delta)) \leq C,
\]
for some $C=C(d,\rho)$. 
By definition,
\[
\mu(\ell_{re_1}(\rho))- \mu(\ell_{re_1}(\rho-\delta))
= \int\limits_{SO_d}\big(\lambda_{d-1}\big(\pi(\ell_{\phi(re_1)}(\rho))\big)- \lambda_{d-1}\big(\pi(\ell_{\phi(re_1)}(\rho-\delta))\big)\big)\,\nu(d\phi).
\]
Note that for any rotation $\phi\in SO_d$, 
the projection $\pi(\ell_x(t))$ of cylinder $\ell_x(t)$ is a $(d-1)$-dimensional cylinder of radius $t$ in $H$ around the line segment $[0,\pi(x)]$. By Lemma~\ref{l:mu-segment-PC}, its $\lambda_{d-1}$-measure is equal to $|\pi(x)|+2t$ for $d=2$ and 
\[
\kappa_{d-1}t^{d-1} + \kappa_{d-2}t^{d-2}\|\pi(x)\| 
= 
\kappa_{d-1}t^{d-1} + \kappa_{d-2}t^{d-2}r\sin\varphi,
\]
for $d\geq 3$, where $\varphi$ is the angle between $x$ and $e_1$. As a result, 
for $d=2$, $\mu(\ell_{re_1}(\rho))- \mu(\ell_{re_1}(\rho-\delta))=2\delta$, and for $d\geq 3$, $\mu(\ell_{re_1}(\rho))- \mu(\ell_{re_1}(\rho-\delta))\leq C\delta r$, for some $C=C(d,\rho)$. Since $\delta(r)=1$ for $d=2$ and $\delta(r) = \frac1r$ for $d\geq 3$, both right hand sides are bounded uniformly in $r$, and the proof is completed. 
\qed

\subsection{Proof of the lower bound of \eqref{intro:visibility-bounds}}\label{sec:PC-lower}

Fix $\alpha>0$ and $\rho>0$. 
Here, we prove that Poisson cylinders model $\mathcal L^\alpha_\rho$ satisfies the lower bound of \eqref{intro:visibility-bounds} with $\delta(r)$ as in part (2) of Theorem~\ref{intro:visibility-mainresult}. 
By Proposition~\ref{prop:visibility-general} and Remark~\ref{rem:visibility-general}, it suffices to show that for some $\delta=D_1(d,\alpha,\rho)\delta(r)$, $s=r$ and $D=D_2(d,\alpha,\rho)$, condition \eqref{eq:visibility-general-lower} holds in dimensions $d\geq 3$ and condition \eqref{eq:visibility-general-lower-2} holds in dimension $d=2$.  
By \eqref{eq:PC-mu-rho}, for $x\in\partial B(r)\cap B(re_1,r)$, 
\[
\mathsf P^\alpha_\rho[0\stackrel{L}\longleftrightarrow x\,|\,0\stackrel{L}\longleftrightarrow re_1]= 
\exp\Big(-\alpha\big(\mu(\ell_{re_1}(\rho)\cup\ell_x(\rho)) - \mu(\ell_{re_1}(\rho))\big)\Big).
\]
Thus, it suffices to prove that there exist $c=c(d,\rho)$ and $C=C(d,\rho)$ such that for all $r\geq 1$, 
\begin{equation}\label{eq:PC-lower-main}
\mu(\ell_{re_1}(\rho)\cup\ell_x(\rho)) - \mu(\ell_{re_1}(\rho))\geq 
\left\{\begin{array}{lll} 
c\tfrac{1}{\delta(r)}\min(\|x-re_1\|,1) - C, && d\geq 3,\\[4pt]
c\tfrac{1}{\delta(r)}\|x-re_1\| - C, && d=2,
\end{array}\right.
\end{equation}
where $\|x-re_1\|\leq r$. 

Using the definition of $\mu$ and the facts that for any rotation $\phi\in SO_d$ and $z\in \R^d$, (a) the rotation of cylinder $\ell_z(\rho)$ by $\phi$ is the cylinder $\ell_{\phi(z)}(\rho)$ and (b) the projection of cylinder $\ell_z(\rho)$ on $H$ is the $(d-1)$-dimensional cylinder of radius $\rho$ around the line segment $[0,\pi(z)]$, which equals to $\ell_{\pi(z)}(\rho)\cap H$, we can rewrite the left hand side of \eqref{eq:PC-lower-main} as 
\begin{multline*}
\int\limits_{SO_d}\Big(\lambda_{d-1}\big((\ell_{\pi(\phi(re_1))}(\rho)\cap H)\cup(\ell_{\pi(\phi(x))}(\rho)\cap H)\big)
- \lambda_{d-1}\big(\ell_{\pi(\phi(re_1))}(\rho)\cap H\big)
\Big)\nu(d\phi)\\
= 
\int\limits_{SO_d} \lambda_{d-1}\big((\ell_{\pi(\phi(x))}(\rho)\setminus\ell_{\pi(\phi(re_1))}(\rho)) \cap H\big)
\nu(d\phi).
\end{multline*}
Let $H_i$ be the hyperplane $\{z=(z_1,\ldots, z_d)\in\R^d\,:\,z_i=0\}$ and $\pi_i$ the orthogonal projection on $H_i$. 
By the invariance of the Haar measure, we have 
\begin{multline*}
\int\limits_{SO_d} \lambda_{d-1}\big((\ell_{\pi(\phi(x))}(\rho)\setminus\ell_{\pi(\phi(re_1))}(\rho)) \cap H\big)
\nu(d\phi)\\
= 
\frac12\int\limits_{SO_d} \lambda_{d-1}\big((\ell_{\pi(\phi(re_1))}(\rho)\Delta\ell_{\pi(\phi(x))}(\rho))\cap H\big)
\nu(d\phi)\\
=
\frac1{2d}\int\limits_{SO_d} \sum\limits_{i=1}^d\lambda_{d-1}\big((\ell_{\pi_i(\phi(re_1))}(\rho)\Delta\ell_{\pi_i(\phi(x))}(\rho))\cap H_i\big)
\nu(d\phi).
\end{multline*}

If $d=2$, then by Lemma~\ref{l:volume-symmetric-difference}, 
\begin{eqnarray*}
\sum\limits_{i=1}^2\lambda_1\big((\ell_{\pi_i(\phi(re_1))}(\rho)\Delta\ell_{\pi_i(\phi(x))}(\rho))\cap H_i\big)
&= &\sum\limits_{i=1}^2 \big|\pi_i(\phi(x))-\pi_i(\phi(re_1))\big|\\
&\geq &\|\phi(x) -\phi(re_1)\|= \|re_1-x\|.
\end{eqnarray*}
Thus, $\mu(\ell_{re_1}(\rho)\cup\ell_x(\rho)) - \mu(\ell_{re_1}(\rho))\geq\frac14\|re_1-x\|$. Since $\delta(r) =1$, \eqref{eq:PC-lower-main} holds. 

\smallskip

Let $d\geq 3$. We will prove that for some $c=c(d,\rho)$ and all $r>8\rho\sqrt d$, $\mu(\ell_{re_1}(\rho)\cup\ell_x(\rho)) - \mu(\ell_{re_1}(\rho))\geq c\tfrac{1}{\delta(r)}\min(\|x-re_1\|,1)$, which implies \eqref{eq:PC-lower-main}.  
Let 
\[
r_i = \min\big(\|\pi_i(\phi(re_1))\|, \|\pi_i(\phi(x))\|\big)
\]
and let $\varphi_i$ be the angle between the vectors $\pi_i(\phi(re_1))$ and $\pi_i(\phi(x))$. By Lemma~\ref{l:volume-symmetric-difference}, if $r_i>8\rho$ then 
\[
\lambda_{d-1}\big((\ell_{\pi_i(\phi(re_1))}(\rho)\Delta\ell_{\pi_i(\phi(x))}(\rho))\cap H_i\big)
\geq c r_i\min\big(r_i\sin\frac{\varphi_i}2,1\big), 
\]
for some $c=c(d,\rho)>0$. By Lemma~\ref{l:vector-projections}, there exists $i_0$ such that 
$\|\pi_{i_0}(\phi(re_1))\|\geq \frac1{\sqrt d}r$, 
$\|\pi_{i_0}(\phi(x))\|\geq \frac1{\sqrt d}r$ 
and $\sin\varphi_{i_0}\geq \frac1{\sqrt d}\sin\varphi$, where we have also used that $\phi$ preserves the norms and angles between vectors. In particular, these imply that $r_{i_0}\geq \frac1{\sqrt d}r>8\rho$ and $\sin\frac{\varphi_{i_0}}2\geq \frac12\sin\varphi_{i_0}\geq \frac{1}{2\sqrt d}\sin\varphi\geq \frac{1}{2\sqrt d}\sin\frac\varphi2$, where in the last inequality we used that $\varphi<\frac\pi2$ (which follows from the assumption $\|x-re_1\|\leq r$). Therefore, 
\[
\lambda_{d-1}\big((\ell_{\pi_{i_0}(\phi(re_1))}(\rho)\Delta\ell_{\pi_{i_0}(\phi(x))}(\rho))\cap H_{i_0}\big)
\geq \frac{c}{2d\sqrt d}\, r\min\big(r\sin\frac\varphi2,1\big).
\]
Since $r\sin\frac\varphi2 = \frac12\|x-re_1\|$, we conclude that for any rotation $\phi\in SO_d$, 
\[
\sum\limits_{i=1}^d\lambda_{d-1}\big((\ell_{\pi_i(\phi(re_1))}(\rho)\Delta\ell_{\pi_i(\phi(x))}(\rho))\cap H_i\big)
\geq \frac{c}{4d\sqrt d}\, r\min\big(\|x-re_1\|,1\big).
\]
Thus, for all $r>8\rho\sqrt d$, $\mu(\ell_{re_1}(\rho)\cup\ell_x(\rho)) - \mu(\ell_{re_1}(\rho))\geq \frac{c}{8d^2\sqrt d}\, r\min\big(\|x-re_1\|,1\big)$. The proof is completed. \qed

\section{Boolean model}\label{sec:BM}

Let $\mathsf Q$ be a probability distribution on $\R_+$ and denote by $\rho$ a generic random variable with distribution $\mathsf Q$. Let $\omega = \sum_{i\geq 1}\delta_{(x_i,r_i)}$ be a Poisson point process on $\R^d\times\R_+$ ($d\geq 2$) with intensity measure $\alpha dx\otimes \mathsf Q$. Its support induces the closed random subset of $\R^d$
\[
\mathcal B^\alpha_\rho = \mathcal B^\alpha_\rho(\omega) =\bigcup\limits_{i\geq 1}B(x_i,r_i),
\]
which is called the \emph{Boolean model with intensity $\alpha$ and radii distribution $\mathsf Q$}. The set $\mathcal B^\alpha_\rho$ does not coincide with $\R^d$ if and only if 
\begin{equation}\label{eq:BM-expected-volume}
\mathsf E[\rho^d]<\infty, 
\end{equation}
see e.g.\ \cite[Proposition~3.1]{MR-Book}, which we may assume from now on, since otherwise the bounds \eqref{intro:visibility-bounds} are trivial. 

The number of balls that intersect a compact set $K$ is a Poisson random variable with parameter 
\[
\int\limits_{(x,r)\,:\,B(x,r)\cap K\neq\emptyset}\alpha dx\otimes \mathsf Q(dr) = \alpha\int\limits_{\R^d}\mathsf P\big[\rho\geq d(x,K)\big]\,dx.
\]
Thus, the law of $\mathcal B^\alpha_\rho$, denoted by $\mathsf P^\alpha_\rho$, is characterized by the relations
\begin{equation}\label{eq:BM-mu}
\mathsf P \big[\mathcal B^\alpha_\rho\cap K = \emptyset\big] = e^{-\alpha\mu_\rho(K)},\quad\text{for compact }K\subset\R^d,
\end{equation}
where 
\[
\mu_\rho(K) = \int\limits_{\R^d}\mathsf P\big[\rho\geq d(x,K)\big]\,dx.
\]

\smallskip

We first compute the value of $\mu_\rho$ for finite cylinders $\ell_{re_1}(\rho)$ in Section~\ref{sec:BM-mu} and then prove part (3) of Theorem~\ref{intro:visibility-mainresult} in Sections~\ref{sec:BM-upper} (upper bound of \eqref{intro:visibility-bounds}) and \ref{sec:BM-lower} (lower bound of \eqref{intro:visibility-bounds}).

\subsection{Auxiliary result}\label{sec:BM-mu}
In this section we compute explicitly $\mu_\rho(\ell_{re_1})$, which gives an explicit expression for $f(r)$ (see \eqref{def:fPvis}). 

\begin{lemma}\label{l:BM-fr}
Let $\alpha>0$ and $\mathsf Q$ a probability law on $\R_+$ satisfying \eqref{eq:BM-expected-volume}. Let $r>0$. Then 
\[
\mu_\rho(\ell_{re_1}) = \kappa_d\mathsf E[\rho^d] + \kappa_{d-1}\mathsf E[\rho^{d-1}]r,
\]
where $\kappa_s$ is the volume of the $s$-dimensional Euclidean unit ball. 
\end{lemma}
\begin{proof}
By the definition of $\mu_\rho$ and Fubini,  
\[
\mu_\rho(\ell_{re_1}) = \int\limits_{\R^d} \mathsf P\big[\rho\geq d(x,\ell_{re_1})\big]\,dx
= \mathsf E\big[\lambda_d\big(\ell_{re_1}(\rho)\big)\big]
\stackrel{(\mathrm{L. }\ref{l:mu-segment-PC})}= \mathsf E\big[\kappa_d\rho^d + \kappa_{d-1}\rho^{d-1}r\big],
\]
where $\lambda_d$ is the Lebesgue measure in $\R^d$, which implies the result.
\end{proof}

\subsection{Proof of the upper bound of \eqref{intro:visibility-bounds}}\label{sec:BM-upper}

Fix $\alpha>0$ and $\mathsf Q$. 
Here, we prove that Boolean model $\mathcal B^\alpha_\rho$ satisfies the upper bound of \eqref{intro:visibility-bounds} with $\delta(r)$ as in part (3) of Theorem~\ref{intro:visibility-mainresult}. 
By Proposition~\ref{prop:visibility-general}, it suffices to show that condition \eqref{eq:visibility-general-upper} holds with $\delta=\delta(r)=\frac1r$ and $D=D(d,\alpha,\mathsf Q)$. 

If $\partial B(r)\cap B(re_1,\delta)$ is visible from $0$ through the vacant set of $\mathcal B^\alpha_\rho$, then the line segment $\ell_{re_1}$ is visible through the vacant set of $\mathcal B^\alpha_{(\rho-\delta)^+}$. Thus, it suffices to prove that there exists $D=D(d,\alpha,\mathsf Q)$ such that for all $r\geq 1$, 
\[
\mathsf P^\alpha_{(\rho-\frac1r)^+}[0\stackrel{L}\longleftrightarrow re_1]\leq D\mathsf P^\alpha_\rho[0\stackrel{L}\longleftrightarrow re_1].
\]
By \eqref{eq:BM-mu}, this is equivalent to 
\[
\mu_\rho(\ell_{re_1})- \mu_{(\rho-\frac1r)^+}(\ell_{re_1}) \leq C,
\]
for some $C=C(d,\mathsf Q)$. The latter is immediate from Lemma~\ref{l:BM-fr}.
\qed

\subsection{Proof of the lower bound of \eqref{intro:visibility-bounds}}\label{sec:BM-lower}

Fix $\alpha>0$ and $\mathsf Q$. 
Here, we prove that Boolean model $\mathcal B^\alpha_\rho$ satisfies the lower bound of \eqref{intro:visibility-bounds} with $\delta(r)$ as in part (3) of Theorem~\ref{intro:visibility-mainresult}. 
By Proposition~\ref{prop:visibility-general}, it suffices to show that condition \eqref{eq:visibility-general-lower} holds with $\delta=D_1(d,\alpha,\mathsf Q)\delta(r)$, $s=2r$ and $D=D_2(d,\alpha,\mathsf Q)$. 
By \eqref{eq:BM-mu}, for $x\in\partial B(r)$, 
\[
\mathsf P^\alpha_\rho[0\stackrel{L}\longleftrightarrow x\,|\,0\stackrel{L}\longleftrightarrow re_1]= 
\exp\Big(-\alpha\big(\mu_\rho(\ell_x\cup\ell_{re_1}) - \mu_\rho(\ell_{re_1})\big)\Big).
\]
Thus, it suffices to prove that there exist $c=c(d,\mathsf Q)$ and $C=C(d,\mathsf Q)$ such that for all $r\geq 1$, 
\[
\mu_\rho(\ell_x\cup\ell_{re_1}) - \mu_\rho(\ell_{re_1})\geq cr\min(\|x-re_1\|,1) - C.
\]
Let $0<r_1<r_2$ be such that $\mathsf P[r_1\leq\rho\leq r_2]>0$. It suffices to prove the inequality for all $r>8r_2$. 
By the definition of $\mu_\rho$ and Fubini, the left hand side is equal to 
\begin{eqnarray*}
\mathsf E\big[\lambda_d\big(\ell_x(\rho)\cup\ell_{re_1}(\rho)\big)\big] 
- \mathsf E\big[\lambda_d\big(\ell_{re_1}(\rho)\big)\big] 
&= &\mathsf E\big[\lambda_d\big(\ell_x(\rho)\setminus\ell_{re_1}(\rho)\big)\big] \\
&\geq &\mathsf E\big[\lambda_d\big(\ell_x(\rho)\setminus\ell_{re_1}(\rho)\big);\,r_1\leq \rho\leq r_2\big] \\
&\stackrel{\eqref{eq:volume-equal-norms}} \geq &\mathsf E\big[cr\rho^{d-2}\min(\|x-re_1\|,\rho);\,r_1\leq \rho\leq r_2\big]\\ 
&\geq &c r\min(\|x-re_1\|,1),
\end{eqnarray*}
for some $c=c(d,\mathsf Q)$. The proof is completed. \qed

\section*{Acknowledgements}
The authors would like to thank the anonymous referees for the careful reading of the paper and useful remarks. The research of both authors has been supported by the DFG Priority Program 2265 ``Random Geometric Systems'' (Project number 443849139).


\begin{thebibliography}{4}

\bibitem{ATT18}
D. Ahlberg, V. Tassion and A. Teixeira (2018) Existence of an unbounded vacant set for subcritical continuum percolation. \emph{Electron. Commun. Probab.} \textbf{23}. 

\bibitem{BJST-visibility-H}
I. Benjamini, J. Jonasson, O. Schramm and J. Tykesson (2009) Visibility to infinity in the hyperbolic plane, despite obstacles. \emph{ALEA Lat. Am. J. Probab. Math. Stat.} \textbf{6}, 323--342.

\bibitem{BT-PC-connected}
E. Broman and J. Tykesson (2016) Connectedness of Poisson cylinders in Euclidean
space. \emph{Annales de l’Institut Henri Poincaré, Probabilités et Statistiques} \textbf{52(1)}, 102--126.

\bibitem{Calka-visibility}
P. Calka, J. Michel and S. Porret-Blanc (2009) Visibilit\'e dans le mod\'ele bool\'een. \emph{C. R. Math. Acad. Sci. Paris} \textbf{347 (11-12)}, 659--662.

\bibitem{DRT-Boolean}
H. Duminil-Copin, A. Raoufi and V. Tassion (2020) Subcritical phase of $d$-dimensional Poisson-Boolean percolation and its vacant set. \emph{Annales Henri Lebesgue} \textbf{3}, 677--700.

\bibitem{ET-visibility}
O. Elias and J. Tykesson (2019)
Visibility in the vacant set of the Brownian interlacements and the Brownian excursion process. \emph{ALEA, Lat. Am. J. Probab. Math. Stat.} \textbf{16}, 1007--1028.

\bibitem{Gouere08}
J.-B. Gou\'er\'e (2008) Subcritical regimes in the Poisson Boolean model of continuum percolation. \emph{Ann. Probab.} \textbf{36}, 1209--1220. 

\bibitem{HST-PC-d3}
M.R. Hil\'ario, V. Sidoravicius and A. Teixeira (2015) Cylinders' percolation in three dimensions. \emph{Probab. Theory Related Fields} \textbf{163(3-4)}, 613--642.

\bibitem{Kahane-1}
J.-P. Kahane (1985) \emph{Some Random Series of Functions} (Camb. Stud. Adv. Math. 5), 2nd edn. Cambridge University Press. 

\bibitem{Kahane-2}
J.-P. Kahane (1990) Recouvrements aléatoires et théorie du potentiel. \emph{Colloq. Math.} \textbf{60/61}, 387--411.

\bibitem{Kahane-3}
J.-P. Kahane (1991) Produits de poids aléatoires indépendants et applications. In \emph{Fractal Geometry and Analysis} (Montreal, PQ, 1989; NATO Adv. Sci. Inst. Ser. C Math. Phys. Sci. \textbf{346}), Kluwer, Dordrecht, 277--324. 

\bibitem{Li-BI}
X. Li (2020) Percolative properties of Brownian interlacements and its vacant set. \emph{J. Theoret. Probab.} \textbf{33(4)}, 1855--1893.

\bibitem{Matheron}
G. Matheron (1975) \emph{Random sets and integral geometry}. John Wiley \& Sons, New York-London-Sydney.

\bibitem{MR-Book}
R. Meester and R. Roy (1996) \emph{Continuum percolation}. Cambridge University Press, Cambridge. 

\bibitem{MP-BM-book}
P. M\"orters and Y. Peres (2010) \emph{Brownian motion}. Cambridge University Press, Cambridge. 

\bibitem{MS-BI}
Y. Mu and A. Sapozhnikov (2024) On questions of uniqueness for the vacant set of Wiener sausages and Brownian interlacements. \emph{Probab. Theory Related Fields} \textbf{190}, 703--751.

\bibitem{Pen18}
M. Penrose (2018) Non-triviality of the vacancy phase transition for the Boolean model. \emph{Electron. Commun. Probab.} \textbf{23}.

\bibitem{Polya-visibility}
G. P\'olya (1918) Zahlentheoretisches und wahrscheinlichkeitstheoretisches \"uber die
sichtweite im walde. \emph{Arch. Math. Phys} \textbf{27}, 135--142.

\bibitem{PortStone}
S.C. Port and C.J. Stone (1978) \emph{Brownian motion and classical potential theory.} Academic Press [Harcourt Brace Jovanovich, Publishers], New York-London.

\bibitem{Sznitman-AM}
A.-S. Sznitman (2010) Vacant set of random interlacements and percolation. \emph{Ann. Math.} \textbf{171}, 2039--2087.

\bibitem{Sznitman-BI}
A.-S. Sznitman (2013) On scaling limits and Brownian interlacements. \emph{Bull. Braz. Math. Soc., New Series}, \textbf{44(4)}, 555--592. Special Issue \emph{IMPA 60 years}.

\bibitem{TC-visibility-H}
J. Tykesson and P. Calka (2013) Asymptotics of visibility in the hyperbolic plane. \emph{Adv. in Appl. Probab.} \textbf{45(2)}, 332--350.

\bibitem{TW-cylinders}
J. Tykesson and D. Windisch (2012) Percolation in the vacant set of Poisson cylinders. 
\emph{Probab. Theory Relat. Fields} \textbf{154}, 165--191.

\bibitem{Zacks-visibility}
S. Zacks (2020). Visibility in Random Fields. In: \emph{The career of a research statistician. Statistics for Industry, Technology, and Engineering.} Birkhäuser, Cham, 97--111.

\end{thebibliography}
\end{document}